\documentclass[10pt,draft]{article}
\usepackage{amsmath,amscd}
\usepackage{amssymb,latexsym,amsthm}
\usepackage{color}
\usepackage[spanish,english]{babel}
\usepackage{amssymb}
\usepackage{color}
\usepackage[latin1]{inputenc}
\usepackage{lscape}
\usepackage{fancyhdr}
\usepackage{pb-diagram}
\usepackage{amsfonts}

\numberwithin{equation}{section}
\newtheorem{theorem}{Theorem}[section]

\newtheorem{definition}[theorem]{Definition}
\newtheorem{proposition}[theorem]{Proposition}
\newtheorem{lemma}[theorem]{Lemma}

\theoremstyle{definition}

\newtheorem{remark}[theorem]{Remark}

\makeatletter
\def\@roman#1{\romannumeral #1}
\makeatother

\title{\textbf{Filtered-graded transfer \\ of noncommutative Gröbner bases}}
\author{Claudia Gallego\\
\texttt{cmgallegoj@unal.edu.co}
\\ Seminario de Álgebra Constructiva - SAC$^2$\\ Departamento de Matemáticas\\ Universidad Nacional de
Colombia, Sede Bogot\'a}
\date{}
\begin{document}
\maketitle
\begin{abstract}
\noindent
As the case of free $\Bbbk$-algebras and $PBW$ algebras, given a bijective skew $PBW$ extension $A$, we will show that it is possible transfer Gröbner bases between $A$ and its associated graded ring.
\bigskip


\bigskip

\end{abstract}
\section{Introduction}
\noindent
In \cite{Li} it was shown that if $A=\Bbbk[a_i]_{i\in \Lambda}$ is a $\Bbbk$-algebra generated by $\{a_i\}_{i\in \Lambda}$ over the field $\Bbbk$, and $I$ a left ideal of $A$, then a nonempty subset $G$ of $I$ is a Gröbner basis for $I$ if, and only if, $\overline{G}$ is a Gröbner basis of $Gr(I)$, where $\overline{G}$ denotes the image of $G$ in $Gr(A)$ and $Gr(I)$ is the left ideal associated to $I$ in $Gr(A)$. A similar fact is proved in \cite{Gomez-Torrecillas2} for the case of $PBW$ algebras. We will present an analogous result for skew $PBW$ extensions,  another class of noncommutative rings and algebras of polynomial type that generalize classical $PBW$ extensions and include many important types of quantum algebras.

\section{Skew $PBW$ extensions}\label{definitionexamplesspbw}
\noindent
In this section we recall the definition of skew $PBW$ (Poincaré-Birkhoff-Witt) extensions defined
firstly in \cite{Gallego2}, and we will review also some basic properties about the polynomial
interpretation of this kind of noncommutative rings. Two particular subclasses of these extensions
are recalled also.
\begin{definition}\label{gpbwextension}
Let $R$ and $A$ be rings. We say that $A$ is a \textit{skew $PBW$ extension of $R$} $($also called
a $\sigma-PBW$ extension of $R$$)$ if the following conditions hold:
\begin{enumerate}
\item[\rm (i)]$R\subseteq A$.
\item[\rm (ii)]There exist finite elements $x_1,\dots ,x_n\in A$ such $A$ is a left $R$-free module with basis
\begin{center}
${\rm Mon}(A):= \{x^{\alpha}=x_1^{\alpha_1}\cdots x_n^{\alpha_n}\mid \alpha=(\alpha_1,\dots
,\alpha_n)\in \mathbb{N}^n\}$.
\end{center}
\item[\rm (iii)]For every $1\leq i\leq n$ and $r\in R-\{0\}$ there exists $c_{i,r}\in R-\{0\}$ such that
\begin{equation}\label{sigmadefinicion1}
x_ir-c_{i,r}x_i\in R.
\end{equation}
\item[\rm (iv)]For every $1\leq i,j\leq n$ there exists $c_{i,j}\in R-\{0\}$ such that
\begin{equation}\label{sigmadefinicion2}
x_jx_i-c_{i,j}x_ix_j\in R+Rx_1+\cdots +Rx_n.
\end{equation}
Under these conditions we will write $A:=\sigma(R)\langle x_1,\dots ,x_n\rangle$.
\end{enumerate}
\end{definition}

A particular case of skew $PBW$ extension is when all derivations $\delta_i$ are zero. Another
interesting case is when all $\sigma_i$ are bijective and the constants $c_{ij}$ are invertible. We
recall the following definition (cf. \cite{Gallego2}).
\begin{definition}\label{sigmapbwderivationtype}
Let $A$ be a skew $PBW$ extension.
\begin{enumerate}
\item[\rm (a)]
$A$ is quasi-commutative if the conditions {\rm(}iii{\rm)} and {\rm(}iv{\rm)} in Definition
\ref{gpbwextension} are replaced by
\begin{enumerate}
\item[\rm (iii')]For every $1\leq i\leq n$ and $r\in R-\{0\}$ there exists $c_{i,r}\in R-\{0\}$ such that
\begin{equation}
x_ir=c_{i,r}x_i.
\end{equation}
\item[\rm (iv')]For every $1\leq i,j\leq n$ there exists $c_{i,j}\in R-\{0\}$ such that
\begin{equation}
x_jx_i=c_{i,j}x_ix_j.
\end{equation}
\end{enumerate}
\item[\rm (b)]$A$ is bijective if $\sigma_i$ is bijective for
every $1\leq i\leq n$ and $c_{i,j}$ is invertible for any $1\leq i<j\leq n$.
\end{enumerate}
\end{definition}
The skew $PBW$ extensions can be characterized in a similar way as it
was done in \cite{Gomez-Torrecillas2} for $PBW$ rings  (see Proposition 2.4 there in).
\begin{theorem}\label{coefficientes}
Let $A$ be a left polynomial ring over $R$ w.r.t.. $\{x_1,\dots,x_n\}$. $A$ is a skew $PBW$
extension of $R$ if and only if the following conditions hold:
\begin{enumerate}
\item[\rm (a)]For every $x^{\alpha}\in Mon(A)$ and every $0\neq
r\in R$ there exist unique elements
$r_{\alpha}:=\sigma^{\alpha}(r)\in R-\{0\}$ and $p_{\alpha ,r}\in
A$ such that
\begin{equation}\label{611}
x^{\alpha}r=r_{\alpha}x^{\alpha}+p_{\alpha , r},
\end{equation}
where $p_{\alpha ,r}=0$ or $\deg(p_{\alpha ,r})<|\alpha|$ if
$p_{\alpha , r}\neq 0$. Moreover, if $r$ is left invertible, then
$r_\alpha$ is left invertible.

\item[\rm (b)]For every $x^{\alpha},x^{\beta}\in Mon(A)$ there
exist unique elements $c_{\alpha,\beta}\in R$ and
$p_{\alpha,\beta}\in A$ such that
\begin{equation}\label{612}
x^{\alpha}x^{\beta}=c_{\alpha,\beta}x^{\alpha+\beta}+p_{\alpha,\beta},
\end{equation}
where $c_{\alpha,\beta}$ is left invertible, $p_{\alpha,\beta}=0$
or $\deg(p_{\alpha,\beta})<|\alpha+\beta|$ if
$p_{\alpha,\beta}\neq 0$.
\end{enumerate}
\end{theorem}
In addition, the skew $PBW$ extensions are filtered rings and its associated graded ring satisfies an interesting property, as shown in the following statement.
\begin{proposition}\label{1.3.2}
Let $A$ be an arbitrary skew $PBW$ extension of $R$. Then, $A$ is a filtered ring with filtration
given by
\begin{equation}\label{eq1.3.1a}
F_m:=\begin{cases} R & {\rm if}\ \ m=0\\ \{f\in A\mid {\rm deg}(f)\le m\} & {\rm if}\ \ m\ge 1
\end{cases}
\end{equation}
and the corresponding graded ring $Gr(A)$ is a quasi-commutative skew $PBW$ extension of $R$.
Moreover, if $A$ is bijective, then $Gr(A)$ is a quasi-commutative bijective skew $PBW$ extension
of $R$.
\end{proposition}
\begin{proof}
See \cite{lezamareyes1}, Theorem 2.2.
\end{proof}
The above proposition enables us proving the Hilbert basis theorem for bijective skew $PBW$ extensions.
\begin{proposition}[Hilbert Basis Theorem]\label{1.3.4}
Let $A$ be a bijective skew $PBW$ extension of $R$. If $R$ is a left $($right$)$ Noetherian ring
then $A$ is also a left $($right$)$ Noetherian ring.
\end{proposition}
\begin{proof}
See \cite{lezamareyes1}, Corollary 2.4.
\end{proof}
\begin{remark}
We developed the Gröbner bases theory for any bijective skew $PBW$ extension. \-Spe\-ci\-fi\-ca\-lly, we established a Buchberger's algorithm for these rings, the computation of syzygies module, as well as some applications as  membership problem, calculation of intersections, quotients, presentation of a module, computing free resolutions, the kernel and image of an homomorphism (see Chapter 5 and Chapter 6 in \cite{Gallego6}, or \cite{Gallego5}). In \cite{lezamapaiba} 
where presented some other applications of this noncommutative Gröbner theory. In all of these works the theory and the applications have been illustrated with many examples.
\end{remark}
\section{For left ideals}
\noindent
In \cite{lezamareyes1} was showed that  if $A$ is a skew $PBW$
extension, then its associated graded ring  $Gr(A)$ is a quasi-commutative skew $PBW$
extension (see Theorem 2.2 there in). In this section we will prove this fact using a different technique. Furthermore, we establish the transfer of Gröbner bases between $A$ and $Gr(A)$. \\
\\
By (\ref{1.3.2}), given $A$ a skew $PBW$ extension of the ring $R$, the collection of subsets $\{F_{p}(A)\}_{p\in \mathbb{Z}}$ of $A$ defined
by
 \[ F_{p}(A):=
    \begin{cases}
     0,  &  \text{if $p\leq -1$},\\
      R,  &  \text{if $p=0$},\\
      \{f\in A|\,\text{deg}(lm(f))\leq p\},  &  \text{if $p\geq 1$.}
      \end{cases} \]
is a filtration for the ring $A$, named \textit{standard filtration}.\\
Now, notice that
\[F_{p}(A)=\bigl\{\sum c_{\alpha}x^{\alpha}\mid c_{\alpha}\in R\setminus \{0\},\, x^{\alpha}\in Mon(A), \, \text{deg}(x^{\alpha})\leq p \bigr\};\]
in this case, we say that this filtration is the \textit{filtration $Mon(A)$-standard on $A$}. Moreover,
\[Mon(A)=\bigcup_{p\geq 0} Mon(A)_{p},\]
where $Mon(A)_{p}:=\{x^{\alpha}\in Mon(A)\mid \text{deg}(x^{\alpha})\leq p\}$, and if $|\alpha|=p$,
then $x^{\alpha}\notin Mon(A)_{p-1}$. In this case, it says that $Mon(A)$ is a \textit{strictly filtered basis}.\\
\\
It can be noted that any filtration $\{F_{p}(A)\}_{p\in \mathbb{Z}}$ on $A$ defines an order function $v:A\to \mathbb{Z}$ in the \-fo\-llo\-wing way:
\begin{center}
  $v(f):= \begin{cases}
      p,  &  \text{if }f\in F_{p}(A)-F_{p-1}(A),\\
      -\infty,  &  \text{if } f\in \cap_{p\in \mathbb{Z}}F_p(A).
      \end{cases}$
\end{center}
\begin{definition}
Let $Gr(A)$ be the graded ring associated to the filtered ring $A$, and let $f\in A$ with $f=\sum_{|\alpha|\leq p} c_{\alpha}x^{\alpha}$,
where $p=$deg$(f)$, $c_{\alpha}\in R\setminus\{0\}$ y $\alpha=(\alpha_{1},\ldots, \alpha_{n})\in \mathbb{N}^{n}$.
In what follows,  $\eta(f)$ will denote \textit{the image} (or \textit{principal symbol}) of $f$ in $Gr(A)$, i.e.,
\[\eta(f):=\sum_{|\alpha|=p} c_{\alpha}x^{\alpha}+F_{p-1}(A)\in F_{p}(A)/F_{p-1}(A).\]
\end{definition}
\begin{lemma}\label{lem1}
Let $A$, $Mon(A)$ and $\{F_{p}(A)\}_{p}$ as above, then:
\begin{enumerate}
\item[\rm (i)] For each $f\in A$, deg$(f)=v(f)$.
\item[\rm (ii)] For each $p\in \mathbb{N}$, $Mon(A)_{p}$ is a $R$-basis for $F_{p}(A)$.
\item[\rm (iii)] For $x^{\alpha}$, $x^{\beta}\in Mon(A)$, $\eta(x^{\alpha})=\eta(x^{\beta})$
if and only if $x^{\alpha}=x^{\beta}$.
\end{enumerate}
\end{lemma}
\begin{proof}
(i) From definition of $\{F_{p}(A)\}_{p\in \mathbb{Z}}$ it follows that if $0\neq f\in A$, then there
exists $p\in \mathbb{N}$ such that $f\in F_{p}(A)-F_{p-1}(A)$ and, therefore, $v(f)=p$. But,
if $f\in F_{p}(A)-F_{p-1}(A)$, then deg$(f)=p$ and we obtain the equality.\\
(ii) Let $f\in F_{p}(A)$, then $f=\sum_{|\alpha|\leq p} c_{\alpha}x^{\alpha}$, and hence, $f\in
\,_{R}\langle Mon(A)_{p}\}$. The linear independence of $Mon(A)_{p}$
it follows from fact that $Mon(A)_{p}\subseteq Mon(A)$ and $Mon(A)$ is linearly independent.\\
(iii) Let $x^{\alpha}$, $x^{\beta}\in Mon(A)$ such that $0\neq\eta(x^{\alpha})=\eta(x^{\beta}) \in
Gr(A)_{p}=F_{p}(A)/F_{p-1}(A)$; this last implies that $x^{\alpha}-x^{\beta}\in
F_{p-1}(A)$, i.e., $x^{\alpha}-x^{\beta}\in \,_{R}\langle Mon(A)_{p-1}\}$. Now, since
$x^{\alpha}$, $x^{\beta}\notin F_{p-1}(A)$, we have that $x^{\alpha}-x^{\beta}=0$,
namely $x^{\alpha}=x^{\beta}$. The other implication is straightforward.
\end{proof}
\begin{lemma}\label{lem2}
If $x^{\alpha}$, $x^{\beta}\in Mon(A)$, with deg$(x^{\alpha})=p$ and deg$(x^{\beta})=q$, then
$\eta(x^{\alpha}x^{\beta})=\eta(x^{\alpha})\eta(x^{\beta})$. In particular, if
$x^{\alpha}=x_{1}^{\alpha_{1}}\cdots x_{n}^{\alpha_{n}}\in F_p(A)-F_{p-1}(A)$, necessarily
$\eta(x^{\alpha})\neq 0$ and $\eta(x^{\alpha})=\eta(x_{1})^{\alpha_{1}}\cdots \eta(x_{n})^{\alpha_{n}}\linebreak\in Gr(A)_{p}$.
\end{lemma}
\begin{proof}
In fact, $x^{\alpha}x^{\beta}=c_{\alpha, \beta}x^{\alpha+\beta}+p_{\alpha, \beta}$, where
$c_{\alpha, \beta}\in R$ is left invertible and $p_{\alpha, \beta}=0$ or deg$(p_{\alpha,
\beta})<|\alpha+\beta|=p+q$ (see Theorem \ref{coefficientes}), whence $0\neq
\eta(x^{\alpha}x^{\beta})=\overline{c_{\alpha,\beta}x^{\alpha+\beta}}=
c_{\alpha,\beta}\overline{x^{\alpha+\beta}}\in F_{p+q}(A)/F_{p+q-1}(A)$. Furthermore, $0\neq
\eta(x^{\alpha})\eta(x^{\beta})=\overline{x^{\alpha}}\overline{x^{\beta}}=
\overline{x^{\alpha}x^{\beta}}\in F_{p+q}(A)/F_{p+q-1}(A)$; but $x^{\alpha}x^{\beta}-c_{\alpha,
\beta}x^{\alpha+\beta}=p_{\alpha, \beta}\in F_{p+q-1}(A)$, then
$\overline{x^{\alpha}x^{\beta}}=\overline{c_{\alpha,\beta}x^{\alpha+\beta}}$, i.e.,
$\eta(x^{\alpha}x^{\beta})=\eta(x^{\alpha})\eta(x^{\beta})$.
\end{proof}
\begin{proposition}\label{basegrad}
Let $A$, $Mon(A)$ and $\{F_{p}(A)\}$ as before, then  $\eta(Mon(A)_{p}):=\{\eta(x^{\alpha})\mid x^{\alpha}\in Mon(A)_{p}\}$,
forms a $R$-basis of $Gr(A)_{p}$ for each $p\in \mathbb{N}$. Moreover,
$\eta(Mon(A)):=\{\eta(x^{\alpha})\mid x^{\alpha}\in Mon(A)\}$ is a $R$-basis for $Gr(A)$.
\end{proposition}
\begin{proof}
Let $f\in F_{p}(A)\setminus F_{p-1}(A)$, then $f=\sum_{|\alpha|\leq
p}c_{\alpha}x^{\alpha}$ with $c_{\alpha}\in R\setminus \{0\}$ y $\eta(f)=
\sum_{|\alpha|=p}c_{\alpha}\eta(x^{\alpha})\neq 0$. By Lemma \ref{lem2}, $\eta(x^{\alpha})\in
Gr(A)_{p}$ for every $\alpha$ with $|\alpha|=p$, thus $\eta(Mon(A)_{p})$ is a generating set
for the left $R$-module $Gr(A)_{p}$. Now, suppose that there are $\lambda_{i}\in R$ such that
$0=\sum \lambda_{i}\eta(x^{\alpha_{i}})\in Gr(A)_{p}$ for certain $x^{\alpha_{i}}\in Mon(A)_{p}$,
then $\sum \lambda_{i}x^{\alpha_{i}}\in F_{p-1}(A)$; but  deg$(x^{\alpha_{i}})=p$ for  each $i$
and $Mon(A)$ is a  $R$-basis filtered strictly, hence $\lambda_{i}=0$ for every $i$.
\end{proof}
The above preliminaries enables us to establish one of the main theorems of this section.
\begin{theorem}\label{grcua}
If $A=\sigma(R)\langle x_{1},\ldots, x_{n}\rangle$ is a (bijective) skew $PBW$ extension of ring $R$, then $Gr(A)$
is a (bijective) quasi-commutative skew $PBW$ extension of $R$.
\end{theorem}
\begin{proof}
We must show that in $Gr(A)$ there exist nonzero elements $y_{1},\ldots, y_{n}$ satisfying the conditions in
(a) from Definition \ref{sigmapbwderivationtype}. Define $y_{i}:=\eta(x_{i})$ for each $1\leq i\leq n$; by
Proposition \ref{basegrad} we have that
\[\eta(Mon(A)):=\{\eta(x^{\alpha})=\eta(x_{1})^{\alpha_{1}}\cdots \eta(x_{n})^{\alpha_{n}}\mid x^{\alpha}\in Mon(A)\}\]
is a $R$-basis for $Gr(A)$. Now, given $r\in R\setminus \{0\}$, there is $c_{i,r}\in R\setminus
\{0\}$ such that $x_{i}r-c_{i,r}x_{i}=p_{i,r}\in R$; from last equality it follows that $\eta(x_{i}r)-\eta(c_{i,r}x_{i})=\eta(p_{i,r})=0$, i.e.,
$\eta(x_{i}r)=\eta(c_{i,r}x_{i})=c_{i,r}\eta(x_{i})$; but $x_ir\neq 0$ for any nonzero $r\in R$ because $Mon(A)$ is
a $R$-basis for the right $R$-module $A_R$ (see \cite{lezamareyes1}, Proposition 1.7), thus
$\eta(x_{i}r)=\eta(x_{i})\eta(r)=\eta(x_{i})r$, and consequently $\eta(x_{i})r=c_{i,r}\eta(x_{i})$.
On the other hand, given $i,j\in \{1,\ldots n\}$, there exists $c_{i,j}\in R\setminus\{0\}$ such that
$x_{j}x_{i}-c_{i,j}x_{i}x_{j}=p_{i,j}\in R+Rx_{1}+\cdots +Rx_{n}$; hence we have that
$\eta(x_{j}x_{i})=\eta(c_{i,j}x_{i}x_{j})=c_{i,j}\eta(x_{i})\eta(x_{j})$, and by Lemma \ref{lem2}
$\eta(x_{j}x_{i})=\eta(x_{j})\eta(x_{i})$,  therefore
$\eta(x_{j})\eta(x_{i})=c_{i,j}\eta(x_{i})\eta(x_{j})$. Since the $c_{i,r}$'s and $c_{i,j}$'s that define
to $Gr(A)$ as a quasi-commutative skew $PBW$ extension are the same that define $A$ as a skew $PBW$ extension
of $R$, then the bijectivity of $A$ implies the of $Gr(A)$.
\end{proof}
\begin{remark}
The last theorem will allow us to establish a back and forth between Gröbner bases theory for $A$ and
$Gr(A)$. As we will show, the existence of one theory implies the existence of the other.
\end{remark}
\noindent
In the following, the set $\eta(Mon(A))$ will be denoted by $Mon(Gr(A))$. Thus, $Mon(Gr(A))$ is the basis for the left $R$-module $Gr(A)$   composed by the standard monomials in the variables $\eta(x_{1}),\ldots, \eta(x_{n})$.\\
\\
We recall the definition of monomial order on a ring.
\begin{definition}\label{monomialorder}
Let $\succeq$ be a total order on $Mon(A)$, it says that $\succeq$
is a monomial order on $Mon(A)$ if the following conditions hold:
\begin{enumerate}
\item[\rm (i)]For every $x^{\beta},x^{\alpha},x^{\gamma},x^{\lambda}\in Mon(A)$
\begin{center}
$x^{\beta}\succeq x^{\alpha}$ $\Rightarrow$
$lm(x^{\gamma}x^{\beta}x^{\lambda})\succeq
lm(x^{\gamma}x^{\alpha}x^{\lambda})$.
\end{center}

\item[\rm (ii)]$x^{\alpha}\succeq 1$, for every $x^{\alpha}\in
Mon(A)$.
\item[\rm (iii)]$\succeq$ is degree compatible, i.e., $|\beta|\geq |\alpha|\Rightarrow x^{\beta}\succeq
x^{\alpha}$.
\end{enumerate}
\end{definition}
Monomial orders are also called \textit{admissible orders}.
\begin{proposition}\label{mongrad}
If $\succeq$ is a monomial order on $Mon(A)$, then relation $\succeq_{gr}$ defined over
$Mon(Gr(A))$ by
\begin{align}\label{ordgrad}
\eta(x^{\alpha})\succeq_{gr} \eta(x^{\beta})\Leftrightarrow x^{\alpha} \succeq x^{\beta}
\end{align}
is a monomial order for $Mon(Gr(A))$.
\end{proposition}
\begin{proof}
We will show that $\succeq_{gr}$  satisfies the conditions in the Definition
\ref{monomialorder}: (i) Let $\eta(x^{\alpha})$, $\eta(x^{\beta})$, $\eta(x^{\lambda})$, $\eta(x^{\gamma})\in Mon(Gr(A))$
and suppose that $\eta(x^{\beta})\succeq_{gr} \eta(x^{\alpha})$,
then,
\begin{center}
$lm(\eta(x^{\gamma})\eta(x^{\beta})\eta(x^{\lambda}))\succeq_{gr}lm(\eta(x^{\gamma})\eta(x^{\alpha})\eta(x^{\lambda}))
\Leftrightarrow
lm(\eta(x^{\gamma}x^{\beta}x^{\lambda}))\succeq_{gr} lm(\eta(x^{\gamma}x^{\alpha}x^{\lambda}))$.
\end{center}
But, $\eta(lm(x^{\gamma}x^{\beta}x^{\lambda}))=lm(\eta(x^{\gamma}x^{\beta}x^{\lambda}))$ for all
$\gamma$, $\beta$, $\lambda\in \mathbb{N}^{n}$: indeed,
$\eta(x^{\gamma}x^{\beta}x^{\lambda})=c\overline{x^{\gamma+\beta+\lambda}}\linebreak =
c\eta(x^{\gamma+\beta+\lambda})$, where $c:=c_{\gamma,\beta}c_{\gamma+\beta,\lambda}$ (see
Remark 8 in \cite{Gallego2}). Therefore,
\begin{center}
$lm(\eta(x^{\gamma}x^{\beta}x^{\lambda}))=lm(c\eta(x^{\gamma+\beta+\lambda}))=\eta(x^{\gamma+\beta+\lambda})
=\eta(lm(x^{\gamma}x^{\beta}x^{\lambda}))$.
\end{center}
Since $\succeq$ is a order monomial on $Mon(A)$, it has $lm(x^{\gamma} x^{\beta}
x^{\lambda}) \succeq lm(x^{\gamma} x^{\alpha} x^{\lambda})$, so that $\eta(lm(x^{\gamma} x^{\beta}
x^{\lambda})) \succeq_{gr} \eta(lm(x^{\gamma} x^{\alpha} x^{\lambda}))$, i.e.,  $lm(\eta
(x^{\gamma} x^{\beta} x^{\lambda}))\succeq_{gr} lm(\eta (x^{\gamma} x^{\alpha} x^{\lambda}))$. In consequence, $lm(\eta(x^{\gamma}) \eta(x^{\beta}) \eta(x^{\lambda})) \succeq_{gr}
lm(\eta(x^{\gamma}) \eta(x^{\alpha})\eta(x^{\lambda}))$.\\
The conditions (ii) y (iii) in Definition \ref{monomialorder} are easily verifiable.
\end{proof}
\begin{lemma}\label{lem3}
Let $A$ as before, $\succeq$ a monomial order on $Mon(A)$ and $f\in A$ an arbitrary element.
Then,
\begin{enumerate}
\item[\rm (i)] $f\in F_{p}(A)$ if and only if deg$(f)\leq p$. Further, $f\in F_{p}(A)- F_{p-1}(A)$
if, and only, if deg$(f)=p$.
\item[\rm (ii)] $\eta(lm(f))=lm(\eta(f))$.
\end{enumerate}
\end{lemma}
\begin{proof}
(i) It follows from the definition of $F_{p}(A)$ and Lemma \ref{lem1}. \\
(ii) Let $f$ be a nonzero polynomial in $A$; there exists $p\in \mathbb{N}$ such that
$f\in F_{p}(A)-F_{p-1}(A)$. Let $f=\sum_{i=1}^{n} \lambda_{i}x^{\alpha_{i}}$, with $\lambda_{i}\in R\setminus
\{0\}$ y $x^{\alpha_{i}}\in Mon(A)_{p}$, $1\leq i\leq n$, where
$x^{\alpha_{1}}\succ x^{\alpha_{2}}\succ\cdots \succ x^{\alpha_{n}}$. Hence,
$lm(f)=x^{\alpha_{1}}$, deg$(f)=p$ and
$\eta(f)=\sum_{|\alpha_{i}|=p}\lambda_{i}\eta(x^{\alpha_{i}})$. From the definition
given for $\succeq_{gr}$, we have that
$lm(\eta(f))=\eta(x^{\alpha_{1}})=\eta(lm(f))$.
\end{proof}
We will prove that the reciprocal of the Proposition \ref{mongrad} also holds.
\begin{proposition}\label{gradorder}
Let $A$ and $Gr(A)$ as before. If  $\succeq_{gr}$ is a monomial order
on $Mon(Gr(A))$, then the relation  $\succeq$ defined as
\begin{align}\label{ordmon}
x^{\alpha}\succ x^{\beta}\Leftrightarrow \eta(x^{\alpha})\succ_{gr} \eta(x^{\beta})
\end{align}
is a monomial order over  $Mon(A)$.
\end{proposition}
\begin{proof}
Since  $\succeq_{gr}$ is a well order, from Definition
\ref{ordmon} it follows that $\succeq$ is a well order too. Now, we show that
$\succeq$ is a monomial order: indeed, let $x^{\alpha}$, $x^{\beta}$, $x^{\gamma}$,
$x^{\lambda}\in Mon(A)$  and suppose that $x^{\beta}\succeq x^{\alpha}$, so:
\[
\begin{cases}
\eta(x^{\beta})\succeq \eta(x^{\alpha})\\
\eta(lm(x^{\gamma}x^{\beta}x^{\lambda}))=lm(\eta(x^{\gamma}x^{\beta}x^{\lambda}))=
lm(\eta(x^{\gamma})\eta(x^{\beta})\eta(x^{\lambda}))\\
\eta(lm(x^{\gamma}x^{\alpha}x^{\lambda}))=lm(\eta(x^{\gamma}x^{\alpha}x^{\lambda}))=
lm(\eta(x^{\gamma})\eta(x^{\alpha})\eta(x^{\lambda}))\\
lm(\eta(x^{\gamma})\eta(x^{\beta})\eta(x^{\lambda}))\succeq_{gr}lm(\eta(x^{\gamma})\eta(x^{\alpha})\eta(x^{\lambda})),
\end{cases}
\]
and hence, $lm(x^{\gamma}x^{\beta}x^{\lambda})\succeq
lm(x^{\gamma}x^{\alpha}x^{\lambda})$. Clearly $x^{\alpha}\succeq 1$ for all $x^{\alpha}\in
Mon(A)$, and $\succeq$ is degree compatible.
\end{proof}

\begin{definition}
Let $I$ be a left (right or two side) ideal of $A$. The \textit{graduation} of
$I$ (or the \textit{associated graded ideal} to $I$) is defined as  $G(I):=\oplus_{p} Gr(I)_{p\in
\mathbb{N}}$, where $Gr(I)_{p}:=I\cap F_{p}(A)/ I\cap F_{p-1}(A)\cong (I+F_{p-1}(A))\cap
F_{p}(A)/F_{p-1}(A) $, for each $p\in \mathbb{N}$; (e.g., see \cite{Nastasescu}).
\end{definition}

Before proceeding, let us recall the definition of Gröbner basis.

\begin{definition}
Let $I\neq 0$ be a left ideal of $A$ and
let $G$ be a non empty finite subset of non-zero polynomials of
$I$, we say that $G$ is a Gröbner basis for $I$ if each element
$0\neq f\in I$ is reducible w.r.t. $G$.
\end{definition}

We have the following characterization for Gröbner bases.
\begin{theorem}\label{teogrobnersigmapbw}
Let $I\neq 0$ be a left ideal of $A$ and
let $G$ be a finite subset of non-zero polynomials of $I$. Then
the following conditions are equivalent:
\begin{enumerate}
\item[\rm(i)]$G$ is a Gröbner basis for $I$.
\item[\rm(ii)]For any polynomial $f\in A$,
\begin{center}
$f\in I$ if and only if $f\xrightarrow{\,\, G\,\, }_{+} 0$.
\end{center}
\item[\rm(iii)]For any $0\neq f\in I$ there exist $g_1,\dots ,g_t\in
G$ such that $lm(g_j)|lm(f)$, $1\leq j\leq t$, {\rm(}i.e., there
exist $\alpha_j\in \mathbb{N}^n$ such that
$\alpha_j+\exp(lm(g_j))=\exp(lm(f))${\rm)} and
\begin{center}
$lc(f)\in \langle \sigma^{\alpha_1}(lc(g_1))c_{\alpha_1,g_1},\dots
,\sigma^{\alpha_t}(lc(g_t))c_{\alpha_t,g_t}\}$.
\end{center}
\end{enumerate}
\end{theorem}
\begin{proof}
See \cite{Gallego2}, Theorem 24.
\end{proof}

\begin{theorem}\label{basesgrad}
Let $A$, $Gr(A)$, $Mon(A)$ and $Mon(Gr(A))$ as before, $\succeq$ a monomial order
over $Mon(A)$, and $I$ a left ideal of $A$. If $\overline{\mathcal{G}}=\{G_{j}\}_{j\in J}$ is a
Gröbner basis for $Gr(I)$, with respect to the monomial order $\succeq_{gr}$, and
such basis is formed by homogeneous elements, then $\mathcal{G}:=\{g_{j}\}_{j\in J}$
is a Gröbner basis for $I$, where $g_{j}\in I$ is a selected polynomial
with property that $\eta(g_{j})=G_{j}$ for each $j\in J$.
\end{theorem}
\begin{proof}
Let $0\neq f\in I \cap F_{p}(A)\setminus F_{p-1}(A)$; we shall show that the condition (iii)
in the Theorem \ref{teogrobnersigmapbw} is satisfied: let $\overline{f}:=\eta(f)$, then $0\neq \overline{f}\in
G(I)_{p}$. Since $\overline{\mathcal{G}}$ is a Gröbner basis of $G(I)$, there exist
$G_{1},\ldots, G_{t}\in \overline{\mathcal{G}}$ such that $lm(G_{j})\mid lm(\overline{f})$ for
each $1\leq j\leq t$ and $lc(\overline{f})\in \langle
\sigma^{\alpha_{1}}(lc(G_{1}))c_{\alpha_{1},G_{1}},\ldots,
\sigma^{\alpha_{t}}(lc(G_{t}))c_{\alpha_{t},G_{t}}\}$, with $\alpha_{j}\in \mathbb{N}^{n}$ such
that $\alpha_{j}+exp(lm(G_{j}))=exp(lm(\overline{f}))=exp(lm(f))=p$ and $c_{\alpha_{j},G_{j}}$ is
the coefficient determined by the product $\eta(x)^{\alpha_{j}}lm(G_{j})$ in $Gr(A)$, for
$1\leq j\leq t$. From this last it follows that $lm(\eta(x)^{\alpha_{j}}lm(G_{j}))=lm(\overline{f})$;
but $lm(\eta(x)^{\alpha_{j}}lm(G_{j}))=lm(\eta(x^{\alpha_{j}}x^{\beta_{j}}))$, where
$x^{\beta_{j}}:=lm(g_{j})$ y $g_{j}\in I \cap F_{p}(A)$ is such that $\eta(g_{j})=G_{j}$.
From Lemma \ref{lem3} we get that $lm(\eta(x^{\alpha_{j}}x^{\beta_{j}}))=
\eta(lm(x^{\alpha_{j}}x^{\beta_{j}})) \in F(A)_{p}/F(A)_{p-1}$, so that
$\eta(lm(x^{\alpha_{j}}x^{\beta_{j}}))=lm(\overline{f})=\eta(lm(f))$. The la\-tter implies that
$lm(x^{\alpha_{j}}x^{\beta_{j}})-lm(f)\in F_{p-1}(A)$ and, therefore,
$lm(x^{\alpha_{j}}x^{\beta_{j}})=lm(f)$, i.e., $lm(g_{j})\mid lm(f)$ for each $1\leq j\leq t$.
Further, $lc(h)=lc(\eta(h))$ for all $h\in A$, then $lc(f)\in \langle
\sigma^{\alpha_{1}}(lc(g_{1}))c_{\alpha_{1},g_{1}},\ldots,
\sigma^{\alpha_{t}}(lc(g_{t}))c_{\alpha_{t},g_{t}}\}$.
\end{proof}
In this way, a Gröbner basis of $Gr(I)$ can be transfer to a Gröbner basis of $I$. In particular, from a Gröbner basis of $Gr(I)$ we can get a set of generators for $I$. Reciprocally, whether we need obtain a generating set of $Gr(I)$  from one of $I=\langle f_{1},\ldots,f_{r}\}$, we could think that $Gr(I)=\langle \eta(f_{1}),\ldots,
\eta(f_{r})\}$. Nevertheless, this a\-ffir\-mation in general is not true: in fact,
let $A=A_{2}(\Bbbk)$, the second Weyl algebra, i.e., $A=\Bbbk[x_{1}, x_{2}][y_{1}, \frac{\partial}{\partial x_{1}}]
[y_{2}, \frac{\partial}{\partial x_{2}}]$ with its associated standard filtration, and
consider the left ideal $I$ generated by $f_{1}=x_{1}y_{1}$  and $f_{2}=x_{2}y_{1}^{2}-y_{1}$.
Note that $x_{1}\in I$, since $x_{1}=(t_{2}x_{1}^{2}-x_{1})f_{1}- (t_{1}x_{1}+2)f_{2}$, but
$\eta(x_{1})\notin \langle \eta(f_{1}), \eta(f_{2})\}$, where
$\eta(f_{1})=\eta(t_{1})\eta(x_{1})\in Gr(I)_{1}$ and
$\eta(f_{2})=\eta(t_{2})\eta(x_{1})^{2}\in Gr(I)_{2}$ (see \cite{Li}). However, if $G=\{f_{1},\ldots,f_{r}\}$
is a Gröbner basis for $I$, we will show that $\eta(G)=\{\eta(f_{1}),\ldots,\eta(f_{r})\}$ is a Gröbner basis
for $Gr(I)$ and, from this we will take a generating set for $Gr(I)$.

\begin{theorem}
With notation as above, let $\mathcal{G}=\{g_{i}\}_{i\in J}$ be a
Gröbner basis for a left ideal $I$ of $A$. Then
$\overline{\mathcal{G}}=\{\eta({g}_{i})\}_{i\in J}$ is a Gröbner
basis of $Gr(I)$ consisting of homogeneous
elements.
\end{theorem}
\begin{proof}
Since $Gr(I)$ is a homogeneous ideal, it suffices to show that every nonzero
homogeneous element $F\in Gr(I)$ satisfies the condition (iii)
in the Theorem \ref{teogrobnersigmapbw}.
Let $0\neq F\in Gr(I)_{p}$, then $F=\eta(f)$ for some
$f\in I\cap F_{p}(A)-I\cap F_{p-1}(A)$ and there exist $g_1,\ldots,g_t\in \mathcal{G}$
with the property that $lm(g_i)\mid lm(f)$ and
$lc(f)\in \langle \sigma^{\alpha_1}(lc(g_1))c_{\alpha_1,g_1},\ldots, \sigma^{\alpha_t}(lc(g_t))c_{\alpha_t,g_t}\}$,
where $\alpha_{i}\in \mathbb{N}^n$ is such that $\alpha_i+\exp(g_{i})=\exp(f)$
for each $1\leq i\leq t$.
By Lemma \ref{lem3} we have that $\eta(lm(f))=lm(\eta(f))=lm(F)$, then
$lm(\eta(g_i))\mid lm(F)$. Further,
since $lc(f)=lc(\eta(f))=lc(F)$, it follows that
$lc(F)\in \langle \sigma^{\alpha_1}(lc(\eta(g_1)))c_{\alpha_1,\eta(g_1)},\ldots,
\sigma^{\alpha_t}(lc(\eta(g_t)))c_{\alpha_t,\eta(g_t)}\}$ and, in consequence
$\overline{\mathcal{G}}$ is a Gröbner basis for $Gr(I)$.
\end{proof}

\section{For modules}
\noindent
Similar results to those presented in the previous section can be proved in the case of modules. For this,
let $M$ be a submodule of the free module $A^{m}$, $m\geq 1$, where
$A$ is a skew $PBW$ extension of a ring $R$. Define the following collection of subsets
of $M$:
\begin{equation}\label{filmod}
F_{p}(M):=\{\boldsymbol{f}\in M\mid \text{deg}(\boldsymbol{f})\leq p\}.
\end{equation}
It is not difficult to show that the collection $\{F_{p}(M)\}_{p\geq 0}$ given in
\ref{filmod} is a filtration for $M$, called the \textit{\textit{natural filtration}} on $M$.
With this filtration we can define the graded module associated to $M$, which will be denoted by
$Gr(M)$, in the following way: $Gr(M):=\oplus_{p\geq 0}F_{p}(M)/F_{p-1}(M)$;
if $\boldsymbol{f}\in F_{p}(M)-F_{p-1}(M)$,
then $\boldsymbol{f}$ is said to have degree $p$. Thus, we may associate to $\boldsymbol{f}$
its \textit{principal symbol}
$\eta(\boldsymbol{f}):=\boldsymbol{f}+F_{p-1}(M)\in G_{p}(M)=F_{p}(M)/F_{p-1}(M)$.
The $Gr(A)$-structure is given by, via distributive laws, the following multiplication:
\begin{equation}
\eta(r)\eta(\boldsymbol{f}):=
\begin{cases}
\eta(r\boldsymbol{f}),  &  \text{if $r\boldsymbol{f}\notin F_{i+j-1}(M)$},\\
0,  &  \text{otherwise}
\end{cases}
\end{equation}
where $r\in F_{i}(A)-F_{i-1}(A)$ and $\boldsymbol{f}\in F_{j}(M)-F_{j-1}(M)$.

Notice that any filtration $\{F_{p}(M)\}_{p\in \mathbb{Z}}$ on $M$ defines an order function
$v:M\to \mathbb{Z}$ in the following way:
\begin{center}
  $v(\boldsymbol{f}):= \begin{cases}
     p,  &  \text{if }\boldsymbol{f}\in F_{p}(M)-F_{p-1}(M),\\
     -\infty,  &  \text{if } \boldsymbol{f}\in \cap_{p\in \mathbb{Z}}F_p(M).
      \end{cases}$
\end{center}

\begin{lemma}\label{lem1mod}
Let $A$, $M$ and $\{F_{p}(M)\}_{p}$ as above. Then for each $\boldsymbol{f}\in M$,
deg$(\boldsymbol{f})=v(\boldsymbol{f})$.
\end{lemma}
\begin{proof}
From definition of $\{F_{p}(M)\}_{p\geq 0}$, it follows that if
$\textbf{0}\neq \boldsymbol{f}\in M$, then there exists $p\in \mathbb{N}$ such that
$\boldsymbol{f}\in F_{p}(M)-F_{p-1}(M)$ and, therefore, $v(\boldsymbol{f})=p$.
But, if $f\in F_{p}(M)-F_{p-1}(M)$, then deg$(\boldsymbol{f})=p$ and we obtain the equality.
\end{proof}
We have a version of the Proposition \ref{mongrad} for module case. For this, remember that
the monomials in $Gr(A)^{m}$ are given by $\overline{\textbf{\emph{X}}}=\eta(\textbf{\emph{X}}):=\eta(x^{\alpha})\overline{\textbf{\emph{e}}}_i$,
where $\overline{\textbf{\emph{e}}}_i$ is a canonical vector of $Gr(A)^{m}$. We also recall the definition of  monomial orders on $Mon(A^m)$.
\begin{definition}\label{monordermod}
A monomial order on $Mon(A^m)$ is a total order $\succeq$
satisfying the following three conditions:
\begin{enumerate}
\item[\rm(i)] $lm(x^{\beta}x^{\alpha})\textbf{e}_{i}\succeq x^{\alpha}\textbf{e}_{i}$, for every
monomial $\textbf{X}=x^{\alpha}\textbf{e}_{i}\in Mon(A^{m})$ and any monomial $x^{\beta}$ in
$Mon(A)$.
\item[\rm(ii)] If $\textbf{Y}=x^{\beta}\textbf{e}_{j}\succeq \textbf{X}=x^{\alpha}\textbf{e}_{i}$, then
$lm(x^{\gamma}x^{\beta})\textbf{e}_{j}\succeq lm(x^{\gamma}x^{\alpha})\textbf{e}_{i}$ for every
monomial $x^{\gamma}\in Mon(A)$.
\item[\rm (iii)]$\succeq$ is degree compatible, i.e., $\deg(\textbf{X})\geq
\deg(\textbf{Y})\Rightarrow \textbf{X}\succeq
\textbf{Y}$.
\end{enumerate}
\end{definition}

\begin{proposition}\label{mongradmod}
If $>$ is a monomial order on $Mon(A^m)$, then relation $>_{gr}$ defined over
$Mon(Gr(A)^m)$ by
\begin{align}\label{ordgradmod}
\eta(\textbf{X})>_{gr} \eta(\textbf{Y})\Leftrightarrow \textbf{X} > \textbf{Y}
\end{align}
is a monomial order for $Mon(Gr(A)^m)$.
\end{proposition}
\begin{proof}
We will show that $\succeq_{gr}$  satisfies the conditions in the Definition
\ref{monordermod}: to begin, note that $>_{gr}$ is a total order because $>$ it is.
Now, to prove (i) we must show that $lm(\eta(x^{\beta})\eta(x^{\alpha}))\overline{\textbf{\emph{e}}}_i \geq_{gr}\eta(x^{\alpha})\overline{\textbf{\emph{e}}}_i$
for every $\overline{\textbf{\emph{X}}}=\eta(x^{\alpha})\overline{\textbf{\emph{e}}}_i\in Mon(Gr(A)^m)$ and
$\eta(x^{\beta})\in Mon(Gr(A))$. It can be noted that,
\begin{center}
$lm(\eta(x^{\beta})\eta(x^{\alpha}))\overline{\textbf{\emph{e}}}_i \geq_{gr}\eta(x^{\alpha})\overline{\textbf{\emph{e}}}_i
\Leftrightarrow
\eta(lm(x^{\beta}x^{\alpha}))\overline{\textbf{\emph{e}}}_i \geq_{gr}\eta(x^{\alpha})\overline{\textbf{\emph{e}}}_i$.
\end{center}
Since $>$ is a monomial order on $Mon(A^m)$, we have that
$lm(x^{\beta}x^{\alpha})\textbf{\emph{e}}_i \geq x^{\alpha}\textbf{\emph{e}}_i$ and, from (\ref{ordgradmod})
it follows that
$\eta(lm(x^{\beta}x^{\alpha}))\overline{\textbf{\emph{e}}}_i \geq_{gr}\eta(x^{\alpha})\overline{\textbf{\emph{e}}}_i$.
So,
$lm(\eta(x^{\beta})\eta(x^{\alpha}))\overline{\textbf{\emph{e}}}_i \geq_{gr}\eta(x^{\alpha})\overline{\textbf{\emph{e}}}_i$.
\\
For (ii), let $\overline{\textbf{\emph{Y}}}=\eta(x^{\beta})\overline{\textbf{\emph{e}}}_j$ and
$\overline{\textbf{\emph{X}}}=\eta(x^{\alpha})\overline{\textbf{\emph{e}}}_i$ monomials
in $Mon(Gr(A)^{m})$ such that $\overline{\textbf{\emph{Y}}} \geq_{gr} \overline{\textbf{\emph{X}}}$. Given
$\eta(x^{\gamma})\in Mon(Gr(A))$, we have
\begin{center}
$lm(\eta(x^{\gamma})\eta(x^{\beta}))\overline{\textbf{\emph{e}}}_j \geq_{gr} lm(\eta(x^{\gamma})\eta(x^{\alpha}))\overline{\textbf{\emph{e}}}_i
\Leftrightarrow
\eta(lm(x^{\gamma}x^{\beta}))\overline{\textbf{\emph{e}}}_j \geq_{gr} \eta(lm(x^{\gamma}x^{\alpha}))\overline{\textbf{\emph{e}}}_i$.
\end{center}
In $Mon(A)$ we get that $lm(x^{\gamma}x^{\beta})\textbf{\emph{e}}_j \geq lm(x^{\gamma}x^{\alpha})\textbf{\emph{e}}_i$
and, once again, from (\ref{ordgradmod}) it follows that \linebreak
$\eta(lm(x^{\gamma}x^{\beta}))\overline{\textbf{\emph{e}}}_j  \geq_{gr} \eta(lm(x^{\gamma}x^{\alpha}))\overline{\textbf{\emph{e}}}_i$.\\
Finally is easily verifiable that $\geq_{gr}$ is degree compatible.
\end{proof}

\begin{lemma}\label{lem3mod}
Let $A$, $M$, $Gr(A)$, $Gr(M)$ and $<$ as before, and consider an arbitrary element $\boldsymbol{f}\in M$.
Then,
\begin{enumerate}
\item[\rm (i)] $\boldsymbol{f}\in F_{p}(M)$ if, and only if, deg$(\boldsymbol{f})\leq p$. Further,
$\boldsymbol{f}\in F_{p}(M)- F_{p-1}(M)$ if, and only, if deg$(\boldsymbol{f})=p$.
\item[\rm (ii)] $\eta(lm(\boldsymbol{f}))=lm(\eta(\boldsymbol{f}))$.
\end{enumerate}
\end{lemma}
\begin{proof}
(i) It follows from the definition of $F_{p}(M)$ and Lemma \ref{lem1mod}. \\
(ii) Let $\boldsymbol{f}$ be a nonzero vector in $M$, then there exists $p\in \mathbb{N}$ such that
$\boldsymbol{f}\in F_{p}(M)-F_{p-1}(M)$. Thus, $\boldsymbol{f}=\sum_{i=1}^{l} \lambda_{i}\boldsymbol{X}_{i}$
with $\lambda_{i}\in R\setminus \{0\}$, $\boldsymbol{X}_{i}\in Mon(A^m)$ where deg$(\boldsymbol{X}_{i})\leq p$
 for each $1\leq i\leq l$, and
$\boldsymbol{X}_{1}>\boldsymbol{X}_{2}>\cdots > \boldsymbol{X}_{l}$. Whence,
$lm(\boldsymbol{f})=\boldsymbol{X}_{1}$ and since deg$(\boldsymbol{f})=p$ and
$\eta(\boldsymbol{f})=\sum_{|\exp(\boldsymbol{X}_i)|=p}\lambda_{i}\eta(\boldsymbol{X}_i)$,
from the definition given for $\geq_{gr}$, we have that
$lm(\eta(\boldsymbol{f}))=\eta(\boldsymbol{X}_{1})=\eta(lm(\boldsymbol{f}))$.
\end{proof}
The conversely of Proposition \ref{mongradmod}  is also true, as it is shown below.
\begin{proposition}
With the same notation used so far, if $\geq_{gr}$ a monomial order on
$Mon(Gr(A)^m)$, then $\geq$ defined as
\begin{align}\label{ordmonmod}
\textbf{X} \geq \textbf{Y}\Leftrightarrow \eta(\textbf{X})\geq_{gr} \eta(\textbf{Y})
\end{align}
is a monomial order over  $Mon(A^m)$.
\end{proposition}
\begin{proof}
Since  $\geq_{gr}$ is a total order, from Definition
\ref{ordmonmod} it follows that $\geq$ is a total order also. Now, we show that
$\geq$ is a monomial order: indeed, let $x^{\beta}\in Mon(A)$ and
$\textbf{\emph{X}}=x^{\alpha}\textbf{\emph{e}}_i$ an element in $Mon(A^m)$;
we must to show $lm(x^{\gamma}x^{\beta})\textbf{\emph{e}}_i\geq x^{\alpha}\textbf{\emph{e}}_i$ for
all $x^{\gamma}\in Mon(A)$; however
\begin{center}
$\eta(lm(x^{\gamma}x^{\beta}))\overline{\textbf{\emph{e}}}_i\geq \eta(x^{\alpha})\overline{\textbf{\emph{e}}}_i
\Leftrightarrow
lm(\eta(x^{\gamma})\eta(x^{\beta}))\overline{\textbf{\emph{e}}}_i\geq x^{\alpha}\overline{\textbf{\emph{e}}}_i$
\end{center}
and since $\geq_{gr}$ is a monomial order, this last inequality is true.
From (\ref{ordmonmod}) it follows that
$lm(x^{\gamma}x^{\beta})\textbf{\emph{e}}_i\geq x^{\alpha}\textbf{\emph{e}}_i$, as we had to show. Now, if $\textbf{\emph{Y}}=x^{\beta}\textbf{\emph{e}}_j$ and
$\textbf{\emph{X}}=x^{\alpha}\textbf{\emph{e}}_i$ are monomials in $Mon(A^{m})$
such that $\textbf{\emph{Y}}\geq \textbf{\emph{X}}$, then
$\eta(\textbf{\emph{Y}})\geq_{gr} \eta(\textbf{\emph{X}})$.
Thus, given $\eta(x^{\gamma})\in Mon(Gr(A))$ we have that
\begin{center}
$lm(\eta(x^{\gamma})\eta(x^{\beta}))\overline{\textbf{\emph{e}}}_j\geq_{gr} lm(\eta(x^{\gamma})\eta(x^{\alpha}))\overline{\textbf{\emph{e}}}_i$
\end{center}
i.e.,
\begin{center}
$\eta(lm(x^{\gamma}x^{\beta}))\overline{\textbf{\emph{e}}}_j\geq_{gr} \eta(lm(x^{\gamma}x^{\alpha}))\overline{\textbf{\emph{e}}}_i$.
\end{center}
This implies that $lm(x^{\gamma}x^{\beta})\textbf{\emph{e}}_j\geq lm(x^{\gamma}x^{\alpha})\textbf{\emph{e}}_i$.
Finally, it is easy to prove that $\geq$ is degree compatible.
\end{proof}
We are ready to prove the main theorem of this last section.
\begin{theorem}
Let $A$, $Gr(A)$, $Mon(A)$ and $Mon(Gr(A))$ be as before, $\geq$ a monomial order
over $Mon(A^m)$, and $M$ a nonzero submodule of $A^m$. The following statements hold:
\begin{enumerate}
\item[\rm (i)]  If $\overline{\mathcal{G}}=\{\textbf{G}_{j}\}_{j\in J}$ is a
Gröbner basis for $Gr(M)$, with respect to the monomial order $\geq_{gr}$, and
such basis is formed by homogeneous elements, then $\mathcal{G}:=\{\textbf{g}_{j}\}_{j\in J}$
is a Gröbner basis for $M$, where $\textbf{g}_{j}\in M$ is a selected vector
with the property that $\eta(\textbf{g}_{j})=\textbf{G}_{j}$ for each $j\in J$.
\item[\rm (ii)] If $\mathcal{G}=\{\textbf{g}_{i}\}_{i\in J}$ is a Gröbner basis for $M$, then
$\overline{\mathcal{G}}=\{\eta(\textbf{g}_{i})\}_{i\in J}$ is a Gröbner  basis of
$Gr(M)$ consisting of homogeneous elements.
\end{enumerate}
\end{theorem}
\begin{proof}
(i) Let $\textbf{0}\neq \boldsymbol{f}\in F_{p}(M)\setminus F_{p-1}(M)$; we shall show that
the condition (iii) in Theorem \ref{teogrobnersigmapbw}, for module case, is satisfied (see \cite{Jimenez2}, Theorem 26): let
$\overline{\boldsymbol{f}}:=\eta(\boldsymbol{f})$, then $\textbf{0}\neq \overline{\boldsymbol{f}}\in
G(M)_{p}$. Since $\overline{\mathcal{G}}$ is a Gröbner basis of $G(M)$, there exist
$\textbf{\emph{G}}_{1},\ldots, \textbf{\emph{G}}_{t}\in \overline{\mathcal{G}}$
such that $lm(\textbf{\emph{G}}_{j})\mid lm(\overline{\boldsymbol{f}})$ for
each $1\leq j\leq t$ and $lc(\overline{\boldsymbol{f}})\in \langle
\sigma^{\alpha_{1}}(lc(\textbf{\emph{G}}_{1}))c_{\alpha_{1},\textbf{\emph{G}}_{1}},\ldots,
\sigma^{\alpha_{t}}(lc(\textbf{\emph{G}}_{t}))c_{\alpha_{t},\textbf{\emph{G}}_{t}}\}$,
with $\alpha_{j}\in \mathbb{N}^{n}$ such that
$\alpha_{j}+\exp(lm(\textbf{\emph{G}}_{j}))=\exp(lm(\overline{\boldsymbol{f}}))=p$
and $c_{\alpha_{j},\textbf{\emph{G}}_{j}}$ is
the coefficient determined by the product $\eta(x)^{\alpha_{j}}lm(\textbf{\emph{G}}_{j})$ in $Gr(M)$, for
$1\leq j\leq t$. But, $\exp(lm(\overline{\boldsymbol{f}}))=\exp(lm(\boldsymbol{f}))$, thus
of the above mentioned follows that
$lm(\eta(x^{\alpha_{j}})lm(\textbf{\emph{G}}_{j}))=lm(\overline{\boldsymbol{f}})$;
note that $lm(\eta(x^{\alpha_{j}})lm(\textbf{\emph{G}}_{j}))=lm(\eta(x^{\alpha_{j}}\boldsymbol{X}_{j}))$, where
$\boldsymbol{X}:=lm(\boldsymbol{g}_{j})$ and $\boldsymbol{g}_{j}\in F_{p}(M)$ is such that
$\eta(\boldsymbol{g}_{j})=\textbf{\emph{G}}_{j}$.
From Lemma \ref{lem3mod} we get that $lm(\eta(x^{\alpha_{j}}\boldsymbol{X}))=
\eta(lm(x^{\alpha_{j}}\boldsymbol{X})) \in F(M)_{p}/F(M)_{p-1}$, so that
$\eta(lm(x^{\alpha_{j}}\boldsymbol{X}))=lm(\overline{\boldsymbol{f}})=\eta(lm(\boldsymbol{f}))$.
The latter implies that
$lm(x^{\alpha_{j}}\boldsymbol{X})-lm(\boldsymbol{f})\in F_{p-1}(M)$ and, therefore,
$lm(x^{\alpha_{j}}\boldsymbol{X})=lm(\boldsymbol{f})$, i.e.,
$lm(\textbf{\emph{g}}_{j})\mid lm(\textbf{\emph{f}})$ for each $1\leq j\leq t$.
Further, $lc(h)=lc(\eta(\boldsymbol{h}))$ for all $\boldsymbol{h}\in A^m$,
then $lc(\boldsymbol{f})\in \langle
\sigma^{\alpha_{1}}(lc(\boldsymbol{g}_{1}))c_{\alpha_{1},\boldsymbol{g}_{1}},\ldots,
\sigma^{\alpha_{t}}(lc(\boldsymbol{g}_{t}))c_{\alpha_{t},\boldsymbol{g}_{t}}\}$.\\

(ii) Since $Gr(M)$ is a graded module, it suffices to show that every nonzero
homogeneous element $\textbf{\emph{F}}\in Gr(M)$ satisfies the condition (iii) in
the Theorem \ref{teogrobnersigmapbw} for module case.
Suppose that $\textbf{\emph{F}}\in Gr(M)_{p}$; then, $\textbf{\emph{F}}=\eta(\boldsymbol{f})$ for some
$\boldsymbol{f}\in F_{p}(M)-F_{p-1}(M)$ and there exist
$\boldsymbol{g}_1,\ldots,\boldsymbol{g}_t\in \mathcal{G}$ with the property that
$lm(\boldsymbol{g}_i)\mid lm(\boldsymbol{f})$
and $lc(\boldsymbol{f})\in \langle \sigma^{\alpha_1}(lc(\boldsymbol{g}_1))c_{\alpha_1,\boldsymbol{g}_1},\ldots,
 \sigma^{\alpha_t}(lc(\boldsymbol{g}_t))c_{\alpha_t,\boldsymbol{g}_t}\}$,
where $\alpha_{i}\in \mathbb{N}^n$ is such that
$\alpha_i+\exp(\boldsymbol{f}_{i})=\exp(\boldsymbol{f})$ for each $1\leq i\leq t$.
By Lemma \ref{lem3mod} we have that $lm(\boldsymbol{f})=lm(\eta(\boldsymbol{f}))=lm(\textbf{\emph{F}})$,
then $lm(\eta(\boldsymbol{g}_i))\mid lm(\textbf{\emph{F}})$ and,
since $lc(\boldsymbol{f})=lc(\eta(\boldsymbol{f}))=lc(\textbf{\emph{F}})$,
it follows that
$lc(\textbf{\emph{F}})\in \langle \sigma^{\alpha_1}(lc(\eta(\textbf{\emph{g}}_1))) c_{\alpha_1,\eta(\textbf{\emph{g}}_1)},\ldots,
\sigma^{\alpha_t}(lc(\eta(\textbf{\emph{g}}_t)))c_{\alpha_t,\eta(\textbf{\emph{g}}_t)}\}$ and, hence,
$\overline{\mathcal{G}}$ is a Gröbner basis for $Gr(M)$.
\end{proof}


\end{document}